\documentclass{article}

\usepackage{amsmath}
\usepackage{amsthm}
\usepackage{amssymb}
\usepackage{mathtools}
\usepackage{mathrsfs}
\usepackage{setspace} 
\usepackage{lscape} 
\usepackage{fancyhdr}
\usepackage[vcentering,dvips]{geometry}	
\usepackage{color}
\usepackage{xcolor}
\usepackage{enumitem}
\usepackage{url} 
\usepackage[inactive]{srcltx} 
\usepackage{relsize} 
\usepackage{graphicx}
\usepackage[pdftex,plainpages=false,pdfpagelabels=true,breaklinks=true,pagebackref]{hyperref} 
\hypersetup{%
	bookmarksnumbered = true,
	pdftitle={\@title},
	pdfauthor={\@author},
	pdfsubject={\@dept},
	pdfkeywords={\@keywords},
	pdfpagelayout=SinglePage,
    bookmarksopen=False,
	pdfborder=0 0 0, 		
    pdffitwindow=true,      
    pdfcreator={\@author},  
    pdfnewwindow=true,      
    colorlinks=true,        
    linkcolor=blue,         
    citecolor=magenta,      
    filecolor=magenta,      
    urlcolor=cyan           
}
\hypersetup{allcolors=blue}

\theoremstyle{plain}
\newtheorem*{theorem*}{Theorem}
\newtheorem{theorem}{Theorem}
\newtheorem{definition}{Definition}

\newtheorem*{corollary*}{Corollary}
\newtheorem{lemma}{Lemma}[theorem]
\newtheorem{notation}{Notation}
\newtheorem*{notation*}{Notation}
\newtheorem{observation}{Observation}
\newtheorem{remark}{Remark}

\DeclareMathOperator{\Exp}{Exp}
\DeclareMathOperator{\AK}{AK}

\title{Absolute Constants of Koras-Russell-like threefolds}
\author{Guillermo Valdeon}
\date{March 2025}

\begin{document}
\maketitle
\begin{center}
    \parbox{\linewidth}{
    \begin{abstract}
    This result generalizes a previous result established in \cite{thesis} where the Absolute Constants of the Koras-Russell threefold was shown to be invariant under translates in the base field to the Absolute Constants of the Koras-Russell threefold being invariant under translates of any of its absolute invariants. It must be pointed out that there is no prior reason for the ring of Absolute Constants to be invariant under so general translates, but remains a possibility that it is part of a more general property of the Absolute Constants ring.
    \end{abstract}
    }
\end{center}
\section{Introduction.}
The Absolute Constants invariant together with the Koras-Russell threefold have been key in solving several conjectures in algebraic geometry. Now this work shows that at least for the Koras-Russell threefold the Absolute Constants invariant remains the same for any translate with respect to any of its invariants, hinting at possibly another general property of the Absolute Constants invariant that can help in its calculation over other affine curves.

\section{Preliminary concepts and results.}
This work presents the results of identifying the ring of absolute constants for translates over the Koras-Russell cubic threefold over its ring of absolute constants. Throughout this work, consider \(F, k\) to be fields.
\begin{notation}
    \(R_{F,\eta}\) denotes the domain \(\frac{F[X,Y,Z,T]}{X^2Y+Z^2+T^3+\eta(X)}\) for \(\eta\in F[X]\). 
\end{notation}
\begin{remark}
    \(\AK(R_{F,\eta}/F)\subseteq F[x]\) for all \(\eta\in F[X]\). 
\end{remark}
\begin{proof}
    Let \(\eta\in F[X]\) be such that \(\eta(X)=\sum_{i=0}^n c_iX^i\) for \(c_i\in F\). Consider a ring homomorphism \(\phi_1\) as in \cite{Hyper}, then \(\phi_1(x^2y+z^2+t^3+\eta(x))
    =\phi_1\left(x^2y+z^2+t^3+\sum_{i=0}^n c_ix^i\right)
    =\phi_1\left(\sum_{i=0}^n c_ix^i\right)+\phi_1^2(x)\phi_1(y)+\phi_1^2(z)+\phi_1^3(t)
    =\sum_{i=0}^n c_i\phi_1^i(x)+x^2(y+2zU-x^2U^2)+(z-x^2U)^2+t^3
    =\sum_{i=0}^n c_ix^i+x^2y+t^3+2x^2zU-x^4U^2+z^2-2x^2zU+x^4U^2
    =\eta(x)+x^2y+z^2+t^3\cong0\); hence, \(\phi_1\) is a well-defined ring homomorphism, so it has an associated locally finite higher derivation \(D\). Just as in \cite{thesis}, \(D\) is iterative over \(R_{k,\eta}\); and hence, \(\phi_1\) is an exponential function of \(R_{F,\eta}/F\). Only \(x\) and \(t\) are fixed by \(\phi_1\) and \(R_{F\eta}\) is a domain, so the ring of \(\phi_1\)-invariants is \(F[x,t]\). Then \(\AK(R_{F,\eta}/F)\subseteq F[x,t]\).\\
    Again consider, \(\phi_2\colon R_{F}\rightarrow R_{F}[U]\) as in \cite{Hyper}, then \(\phi_2(\eta(x)+x^2y+z^2+t^3)=\phi_2(\eta(x))+\phi_2^2(x)\phi_2(y)+\phi_2^2(z)+\phi_2^3(t)+\phi_2\left(\sum_{i=0}^n c_ix^i\right)=\sum_{i=0}^n c_i\phi_2^i(x)+x^2(y+3t^2U-3x^2tU^2+x^4U^3)+z^2+(t-x^2U)^3=\sum_{i=0}^n c_ix^i+x^2y+z^2+3x^2t^2U-3x^4tU^2+x^6U^2+t^3-3x^2t^2U+3x^4tU^2-x^6U^3=\eta(x)+x^2y+z^2+t^3\cong0\); hence, \(\phi_2\) is a well-defined ring homomorphism, so it has an associated locally finite higher derivation \(D\). \(D\) is iterative over \(R_{k,\eta}\); and hence, \(\phi_2\) is an exponential function. The ring of \(\phi_2\)-invariants is \(F[x,z]\). Therefore, \(\AK(R_{F,\eta}/F)\subseteq F[x,z]\cap F[x,t] = F[x]\).
\end{proof}
\begin{observation}
Consider indeterminates \(X,Y,Z,T\), \(B:=k[X,Y,Z,T]\), \(c\in k\) and \(\eta_c\in k[X]\) such that the indepentent coefficient of \(\eta_c\) is \(c\), let \(I_{\eta_c}\), \(J_c\) be the ideals of \(B\) generated by \(f_{\eta_c}:=X^2Y+Z^2+T^3+\eta_c(X)\), and \(f_c=X^2Y+Z^2+T^3+c\in B\) respectively. Let \(R_{\eta_c}:=B/I_{\eta_c}\), and \(S_c:=B/J_c\). \(I_{\eta_c}\), \(J_c\) are prime ideals of \(B\) for all \(c\in k\) and \(\eta_c\in k[X]\) because they are linear in \(Y\). Let \(x_{\eta_c},y_{\eta_c},z_{\eta_c},t_{\eta_c}\in R_{\eta_c}\) denote the images of \(X,Y,Z,T\in B\) under the canonical map \(B\rightarrow R_{\eta_c}\). Let \(\Xi(x_c),\Xi(y_c),\Xi(z_c),\Xi(t_c)\in S_c\) denote the images of \(X,Y,Z,T\in B\) under the canonical map \(B\rightarrow S_c\). Note that according to the given notation \(R_X\) is the Koras-Russell threefold discussed in \cite{Hyper}.
\end{observation}

\begin{definition}
Let \(\omega_1\) be the grading \(\{B_m\}_{m\in\mathbb{Z}}\) on \(B\) such that for \(m\in\mathbb{Z}\),
\(B_m=\{X^{i_1}Y^{i_2}Z^{i_3}T^{i_4}\,|\,-i_1+2i_2=m\}\).
Also, let \(R_{m,\eta_c}\) be the image of \(B_m\) under the canonical map \(B\rightarrow R_{\eta_c}\) and let \(\mathcal{F}_1:=\left\{\cup_{i\leq m}R_{i, \eta_c}\right\}_{m\in\mathbb{Z}}\).
\end{definition}
\begin{definition}
Let \(\omega_2\) be the grading \(\{B_m\}_{m\in\mathbb{Z}}\) on \(B\) such that for \(m\in\mathbb{Z}\),
\(B_m=\{X^{i_1}Y^{i_2}Z^{i_3}T^{i_4}\,|\,6i_1-6i_2+3i_3+2i_4=m\}\).
Also, let \(S_{m,c}\) be the image of \(B_m\) under the canonical map \(B\rightarrow S_c\) and let \(\mathcal{F}_2:=\{\cup_{i\leq m}S_{i,c}\}_{m\in\mathbb{Z}}\).
\end{definition}

\begin{lemma}\label{filtrations}
With the above notation, the following holds for all \(c\in k\).
\begin{enumerate}
    \item \(\mathcal{F}_1\) is a degree filtration on \(R_{\eta_c}\) having its associated graded ring naturally \(k\)-isomorphic to \(S_c\).
    \item \(\mathcal{F}_2\) is a degree filtration on \(S_c\) having its associated graded ring naturally \(k\)-isomorphic to \(S_0\).
    \item \(\omega_2\) induces a \(\mathbb{Z}\)-grading on \(S_0\) and an \(\mathbb{N}\)-grading on \(k[\Xi(z_0),\Xi(t_0)]\).
\end{enumerate}
\end{lemma}
\begin{proof}
\begin{enumerate}[wide, labelwidth=!, labelindent=0pt]
    \item Under \(\omega_1\), the leading form of \(f_{\eta_c}\) is \(f_c\) a grade \(0\) homogeneous polynomial over the grading \(\omega_1\). Since \(f_c\) is a prime element of \(B\) for all \(c\in k\), \(\mathcal{F}_1\) is a degree filtration on \(R_{\eta_c}\). Identifying the \(\mathcal{F}_1\)-leading forms of \(x_{\eta_c},y_{\eta_c},z_{\eta_c},t_{\eta_c}\in R_{\eta_c}\) with \(\Xi(x_c),\Xi(y_c),\Xi(z_c),\Xi(t_c)\in S_c\) respectively, yields a natural \(k\)-isomorphism of the associated graded ring of \(\mathcal{F}_1\) onto the integral domain \(S_c\).
    \item Under \(\omega_2\), the leading form of \(f_c\) is \(f_0\) a grade \(6\) homogeneous polynomial over the grading \(\omega_2\). Since \(f_0\) is a prime element of \(B\) for all \(c\in k\), \(\mathcal{F}_2\) is a degree filtration on \(S_c\). Identifying the \(\mathcal{F}_2\)-leading forms of \(\Xi(x_c),\Xi(y_c),\Xi(z_c),\Xi(t_c)\in S_c\) with \(\Xi(x_0),\Xi(y_0),\Xi(z_0),\Xi(t_0)\in S_0\) respectively, yields a natural \(k\)-isomorphism of the associated graded ring of \(\mathcal{F}_2\) onto the integral domain \(S_0\).
    \item Finally, since \(f_0\) is homogeneous under \(\omega_1\), and \(\omega_2\); it induces a \(\mathbb{Z}\)-grading on \(S_0\) and clearly, induces an \(\mathbb{N}\)-grading on the subring \(k[\Xi(z_0),\Xi(t_0)]\).
\end{enumerate}
\end{proof}
\begin{lemma}\label{trick}
Let \(A:=B/N\) where \(N\) is a prime ideal of \(B\) generated by \(Q:=X^2Y+g\) where \(g\in k[X,Z,T]\). Let \(x,y,z,t\in A\) denote the images of \(X,Y,Z,T\) under the canonical map \(B\rightarrow A\). If \(0\neq p\in A\), then there exist \(\varepsilon(p)\in\mathbb{N}\), \(u_j\), \(v_j\in k[z,t]\) for \(1\leq j\leq \varepsilon(p)\) and \(h\in k[x,z,t]\) such that
\(p=h+\sum_{j=1}^{\varepsilon(p)}(u_jx+v_j)y^j\) where \((u_{\varepsilon(p)},v_{\varepsilon(p)})\neq(0,0)\).
\end{lemma}
\begin{proof}
Clearly \(A=k[x,y,z,t]\). Consider a product \(\alpha y^i\in A\) where \(1\leq i\in\mathbb{N}\) and \(\alpha\in k[x,z,t]\). Then \(\alpha=\beta x^2+ux+v\) for some \(\beta\in k[x,z,t]\), \(u,v\in k[z,t]\). Since \(x^2y=-g\in k[x,z,t]\), so \(\alpha y^i=(ux+v)y^i+\gamma y^{i-1}\) where \(\gamma:=-g\beta\in k[x,z,t]\). Now the assertion readily follows by linearity.
\end{proof}

\section{Main Result.}
\begin{theorem}
Fix \(c\in k\), and \(\eta_c\in k[X]\) let \(R:=R_{\eta_c}\) and let \(\phi\in\Exp(R/k)\). If \(R^{\phi}\neq R\), then \(k[x]\subseteq R^{\phi}\subseteq k[x,z,t]\).
\end{theorem}
\begin{proof}
Fix \(c\in k\) and \(\eta_c\in k[X]\). Let \(S:=S_c,x:=x_c,y:=y_c,z:=z_c,t:=t_c\)
By (i) of \ref{filtrations}. \(\mathcal{F}_1\) is a degree filtration on \(R\). Let \(\delta\) denote the corresponding degree function on \(R\). Suppose, if possible, \(R^{\phi}\setminus{k[x,z,t]}\neq\emptyset\). Let \(p\in R^{\phi}\setminus{k[x,z,t]}\). Applying \ref{trick}, there are \(\varepsilon(p)\in\mathbb{N}\), \(u_j,v_j,\in k[z,t]\) for \(1\leq j\leq \varepsilon(p)\) and \(h\in k[x,z,t]\) such that
\(p=h+\sum_{j=1}^{\varepsilon(p)}(u_jx+v_j)y^j\) where \((u_{\varepsilon(p)},v_{\varepsilon(p)})\neq(0,0)\).
Our choice of \(p\) forces \(\varepsilon(p)\geq1\). Let \(m:=\varepsilon(p)\), \(u:=u_m\) and \(v:=v_m\). Note that \(\delta((ux+v)y^m)=2m\) if \(v\neq0\) and \(\delta((ux+v)y^m)=2m-1\) otherwise. Since \(\delta((u_jx+v_j)y^j)\leq2j\leq2m-2\) for \(0\leq j<m\), then \(\delta((ux+v)y^m)=\delta(p)\geq1\). So, the \(\mathcal{F}_1\)-leading form \(\Xi(p)\in S\) of \(p\) is either \(\Xi(u)\bar{x}\bar{y}^m\) or \(\Xi(v)\bar{y}^m\) where \(\Xi(u),\Xi(v)\) denote the leading forms of \(u,v\) respectively. \(\Xi(u),\Xi(v)\in k[\bar{z},\bar{t}]\). It is claimed there is \(p\in R^{\phi}\) such that \(\Xi(p)\) is either \(\Xi(u)\bar{x}\bar{y}^m\) or \(\Xi(v)\bar{y}^m\) with \(m\in\mathbb{N}\) and \(\Xi(v)\in k[\bar{z},\bar{t}]\setminus{k}\). Supposing otherwise, i.e. supposing \(\Xi(p)=\Xi(v)\bar{y}^m\) with \(0\neq\Xi(v)\in k\) for all \(0\neq p\in R^{\phi}\), derives to a contradiction.

First suppose there is \(v\in R^{\phi}\cap k[x,z,t]\setminus{k}\). Say \(v=v_0+\sum_{i\geq1}v_ix^i\) where \(v_i\in k[z,t]\) for \(i\geq0\). If \(v_0\in k[z,t]\setminus{k}\), then clearly \(\Xi(v)=\Xi(v_0)\in k[\bar{z},\bar{t}]\setminus{k}\) contrary to our supposition. If \(v_0\in k\),  then \(0\neq v-v_0\in R^{\phi}\) and since \(R^{\phi}\) is factorially closed, \(x\in R^{\phi}\). But then, \(\Xi(x)=\Xi(1)\bar{x}\bar{y}^0\) contrary to our supposition above. Hence \(R^{\phi}\cap k[x,z,t]=k\). Then, for \(\varepsilon(p)\) and \(v\) from the expression of \(p\) above, \(\varepsilon(p)\geq1\) and \(0\neq v\in k\), for all \(0\neq p\in R^{\phi}\). In particular, for each \(p\in R^{\phi}\setminus{k}\), \(\delta(p)=2\varepsilon(p)\geq2\). Since the transcendence degree of \(R\) over \(R^{\phi}\) is \(1\), the elements \(z,t\in R\) must be algebraically dependent over \(R^{\phi}\). Let \(U,V\) be indeterminates and \(0\neq P\in R^{\phi}[U,V]\) be such that \(P(z,t)=0\). The polynomial \(P\) is all not in \(k[U,V]\) since \(z,t\) are algebraically independent over \(k\). Let \(P(U,V)=\sum p(i,j)U^iV^j\) with \(p(i,j)\in R^{\phi}\) for all \(i,j\). Then, there is at least one \(i,j\) such that \(p(i,j)\) is not in \(k\) and hence \(\delta(p(i,j))\geq2\). Partition \(\textit{suppt}(P):=\{(i,j)\,|\,p(i,j)\neq0\}\) into sets \(M,M^{'}\) such that \(\delta(p(i,j))=2n:=\max\{\delta(p(i,j))\in\textit{suppt}(P)\}\) for all \((i,j)\in M\) and \(\delta(p(i,j))<2n\) for all \((i,j)\in M^{'}\). A straightforward verification shows that
\(P^{'}(U,V)+\sum_{(i,j)\in M\setminus{M^{'}}}p(i,j)U^iV^j=P(U,V)
=\mu(U,V)y^n+\eta(U,V)xy^n+\sum_{r=1}^{n}\nu_r(U,V)y^{n-r}\)
where \(0\neq\mu(U,V)\in k[U,V]\), \(\eta(U,V)\in k[z,t][U,V]\),
\(P^{'}(U,V):=\sum_{(i,j)\in M^{'}}p(i,j)U^iV^j\), and for \(1\leq r\leq n\), \(\nu_r(U,V)\in k[x,z,t][U,V]\).

Now \(P(z,t)=0\) and since \(z,t\) are algebraically independent over \(k\), \(\mu(z,t)\neq0\). So,

\(0\neq\mu(z,t)y^n=-\eta(z,t)xy^n-\sum_{r=1}^{n}\nu_r(z,t)y^{n-r}-P^{'}(z,t)\).

This is absurd since \(\delta(RHS)<2n\) whereas \(\delta(\mu(z,t)y^n)=2n\). In conclusion, there is \(p\in R^{\phi}\) such that \(\Xi(p)\) is either \(\Xi(u)\bar{x}\bar{y}^m\) or \(\Xi(v)\bar{y}^m\) with \(m\in\mathbb{N}\) and \(\Xi(v)\in k[\bar{z},\bar{t}]\setminus{k}\).

Next, let \(\bar{\phi}\in\Exp(S/k)\) be the nontrivial grade-preserving exponential map induced by \(\phi\). The, for \(p\in R^{\phi}\setminus{k[x,z,t]}\), so \(\Xi(p)\in S^{\bar{\phi}}\) and either \(\Xi(p)=\Xi(u)\bar{x}\bar{y}^m\) or \(\Xi(v)\bar{y}^m\) for some positive integer \(m\). Since \(S^{\bar{\phi}}\) is factorially closed in \(S\), surely \(\bar{y}\in S^{\bar{\phi}}\). Furthermore, let \(p\in R^{\phi}\) such that either \(\Xi(p)=\Xi(u)\bar{x}\bar{y}^m\) or \(\Xi(p)=\Xi(v)\bar{y}^m\) with \(m\geq0\) and \(\Xi(v)\in k[\bar{z},\bar{t}]\setminus{k}\). It readily follows that either \(\bar{x}\in S^{\bar{\phi}}\) or \(\Xi(v)\in S^{\bar{\phi}}\setminus{k}\). Therefore either \(k[\bar{x},\bar{y}]\subseteq S^{\bar{\phi}}\) or \(k[\bar{z},\bar{t}]\subseteq S^{\bar{\phi}}\neq k\). Neither of these is a possibility in view of Lemma 4.0.4 and Theorem 4.1 of \cite{thesis}. Thus \(R^{\phi}\subseteq k[x,z,t]\).

Lastly, suppose \(x\) is not in \(R^{\phi}\). Let \(q,h\in R^{\phi}\subseteq k[x,z,t]\) be algebraically independent over \(k\). Then, \(q=q_1x+q_2\), \(h=h_1x+h_2\) for \(q_1,h_1\in k[x,z,t]\) and \(q_2,h_2\in k[z,t]\). If there is \(0\neq\rho(X,Y)\in k[X,Y]\) with \(\rho(q_2,h_2)=0\), then \(\rho(q,h)=x\sigma\) for some \(\sigma\in k[x,z,t]\) and that forces \(x,\sigma\in R^{\phi}\) contradicting our supposition. So, \(q_2,h_2\) are algebraically independent over \(k\). Furthermore, letting \(\bar{\phi}\in\Exp(S/k)\) denote the nontrivial grade-preserving exponential map induced by \(\phi\), then \(k[\Xi(q),\Xi(h)]\subseteq S^{\bar{\phi}}\). Now the elements \(\bar{z},\bar{t}\) of \(S\) both are algebraic over \(k[\Xi(q),\Xi(h)]\) and also algebraic over \(S^{\bar{\phi}}\). Since \(S^{\bar{\phi}}\) is algebraically closed in \(S\), so \(k[\bar{z},\bar{t}]\subseteq S^{\bar{\phi}}\). Now the equation \(\bar{x}^2\bar{y}=-\bar{z}^2-\bar{t}^3-c\) allows us to conclude that each of \(\bar{x},\bar{y}\) is also in \(S^{\bar{\phi}}\). Consequently, \(S=k[\bar{x},\bar{y},\bar{z},\bar{t}]\subseteq S^{\bar{\phi}}\), i.e., \(S^{\bar{\phi}}=S\). This is absurd since \(\bar{\phi}\) is nontrivial. Hence \(k[x]\subseteq R^{\phi}\).
\end{proof}
\begin{observation}
    The previous result states that \(\AK(R_{F,\eta})=F[x]\) for all \(\eta\in F[X]\cong\AK(R_0)\). So, for \(R_{F,0}\) it is obtained that its absolute constants is invariant under translates from its elements.
\end{observation}

\section{Questions for future research.}
It remains for future research the question of how general is the invariance of the absolute constants ring with respect to translates.

\bibliographystyle{plain}
\bibliography{references/references-article}

\end{document}